\documentclass[12pt]{extarticle}
\usepackage[english]{babel}
\usepackage{newlfont}
\usepackage{amsbsy}
\usepackage[dvips]{graphicx}
\usepackage{color}
\usepackage{bbm}
\usepackage{graphicx, subfigure}
\usepackage{color}
\usepackage[center]{caption2}
\usepackage{amssymb}
\usepackage{amsmath}
\usepackage{latexsym}
\usepackage{amsthm}
\usepackage{amsthm}
\theoremstyle{plain}
\newtheorem{thm}{Theorem}
\newtheorem{cor}[thm]{Corollary}
\newtheorem{lem}[thm]{Lemma}

\newtheorem{rem}[thm]{Remark}
\newtheorem{defin}[thm]{Definition}
\newcommand{\R}{\mathbb{R}}

\newcommand{\N}{\mathbb{N}}

\usepackage{mathtools}
\def\multiset#1#2{\ensuremath{\left(\kern-.2em\left(\genfrac{}{}{0pt}{}{#1}{#2}\right)\kern-.2em\right)}}
\usepackage[center]{caption2}

\begin{document}

\title{On rank range of interval matrices}
\author{Elena Rubei}
\date{}
\maketitle

{\footnotesize\em Dipartimento di Matematica e Informatica ``U. Dini'', 
viale Morgagni 67/A,
50134  Firenze, Italia }

{\footnotesize\em
E-mail address: elena.rubei@unifi.it}

\def\thefootnote{}
\footnotetext{ \hspace*{-0.36cm}

}

\begin{abstract} 
An interval matrix is a  matrix 
whose entries are intervals in $\R$.
Let $p , q \in \N-\{0\}$ and let   $\mu =( [m_{i,j}, M_{i,j}])_{i,j}$ be a $p \times q$ interval matrix; 
given a $p \times q$ matrix $A$   with entries in $\R$, we say that  $ A \in \mu $ if $a_{i,j} \in [m_{i,j}, M_{i,j}] $ for any $i,j$.  We establish a criterion to say if an interval matrix contains a matrix of rank $1$.
Moreover we determine the maximum rank of the matrices contained in a given interval matrix.  Finally, for any interval matrix  $\mu$ with no more than $3$ columns,  we describe a way to find the  
range of the ranks of the matrices contained 
in $\mu$.
\end{abstract}

\section{Introduction}

Let $p , q \in \N-\{0\}$;
a    $p \times q$  interval matrix is a  $p \times q$  matrix 
whose entries are intervals in $\R$; let   $\mu =( [m_{i,j}, M_{i,j}])_{i,j}$ be  a  $p \times q$  interval matrix, where
 $m_{i,j}, M_{i,j}$ are real numbers with $m_{i,j} \leq M_{i,j}$  for any $i$ and $j$; 
given a  $p \times q$  matrix $A$  with entries in $\R$, we say that  $ A \in \mu $ if $a_{i,j} \in [m_{i,j}, M_{i,j}] $ for any $i,j$. 
In this paper we investigate about the range of the ranks of the matrices contained in $\mu$.

There are several papers studying when an interval $p \times q $ matrix $\mu$ has full rank, that is when all the matrices contained in $\mu$ have rank equal to $\min\{p,q\}$. 
For any  $p \times q $ interval matrix $\mu = ( [m_{i,j}, M_{i,j}])_{i,j}$ with $m_{i,j} \leq M_{i,j}$,
let $C^{\mu}$,$\Delta^{\mu}$ and $|\mu|$ be 
 the  $p \times q$   matrices such that $$ C^{\mu}_{i,j}= \frac{m_{i,j}+ M_{i,j}}{2}, \hspace*{1cm} \Delta^{\mu}_{i,j}= \frac{M_{i,j}- m_{i,j}}{2}, \hspace*{1cm} 
|\mu|_{i,j} = \max\{|m_{i,j}|,| M_{i,j}|\} $$ for any $i,j$. 
The following theorem
characterizes full-rank square interval matrices:

\begin{thm} {\bf (Rohn, \cite{Rohn})} 
Let   $\mu =( [m_{i,j}, M_{i,j}])_{i,j}$ be  a  $p \times p$  interval matrix, where
 $m_{i,j} \leq M_{i,j}$  for any $i,j$.
Let $Y_p=\{-1,1\}^p$ and, for any $x
 \in Y_p$, denote by $T_x$ the diagonal matrix whose diagonal is $x$; moreover
 define
  $$C^{\mu}_{x,y} = C^{\mu} - T_x \, \Delta^{\mu} \, T_y.$$

Then  $\mu$ is a full-rank interval matrix 
if and only if,  for each
$x,y \in Y_p$,  $$det(C^{\mu})det(C^{\mu}_{x,y})>0. $$
\end{thm}

See  \cite{Rohn} and  \cite{Rohn2} for other characterizations. 
The following theorem characterizes 
full-rank $p \times q$ interval matrices, see
\cite{Rohn3}, \cite{Rohn4}, \cite{Shary}:

\begin{thm}  {\bf (Rohn)} 
A $p \times q$ interval matrix $\mu$ with $p \geq q$ has full rank if and only if the system
of inequalities $$\hspace*{2cm} |C^{\mu} x| \leq \Delta^{\mu}
|x|, \hspace*{1.5cm} x \in \R^q$$ has only the trivial solution $x=0$. 
\end{thm}

  A problem which can be connected with the quoted ones
  is the one of the partial matrices: let $K$ be a field; a partial matrix over $K$ is a 
matrix where only some of the entries are given and they are elements of 
$K$;
a completion of a partial  matrix is a specification of the unspecified entries. 
We say that a submatrix of a partial matrix is specified if all its entries are given.
The problem of determining whether, given a partial 
matrix, a completion  with some prescribed property exists and 
related problems have been widely studied: we quote, for instance, 
the papers 
 \cite{CD}, \cite{CJRW}, 
\cite{Woe}. In particular there is a wide literature about
the problem of determining the maximal and the minimal rank of the completions of a partial matrix.

In \cite{CJRW}, Cohen,  Johnson,  Rodman and Woerdeman determined 
the maximal rank of the completions of a partial
matrix in terms of the ranks and the sizes of its maximal  
specified submatrices; see also \cite{CD}
for the proof.
 The problem of determining the minimal rank of   the completions of a partial matrix seems more difficult  and it has been solved only in some particular cases,
  see for instance the paper \cite{Woe1} for the case of triangular matrices. We quote also the paper \cite{HHW}, where the authors estabish a criterion to say if a partial matrix has a completion of rank $1$. 
  
Also for interval matrices, the problem of determining the miminal rank of the matrices contained in 
a given interval matrix seems much more difficult than the problem  of determining the 
maximal rank.
In this paper
we establish a criterion to say if an interval matrix contains a matrix of rank $1$; the criterion is analogous to the one to 
  estabish if a partial matrix has a completion of rank $1$ found in \cite{HHW}.
Moreover we determine the maximum rank of the matrices contained in a given interval matrix.  Finally, for any interval matrix  $\mu$ with no more than $3$ columns,  we describe a way to find the  
range of the ranks of the matrices contained 
in $\mu$.

\section{Notation and first remarks}

$\bullet$  Let $\R_{>0}$ be the set $\{x \in \R | \; x >0\}$ and
let $\R_{\geq 0}$ be the set $\{x \in \R | \; x  \geq 0\}$; we define analogously $\R_{<0}$ and $\R_{ \leq 0}$.

$\bullet $ Throughout the paper  let $p , q \in \N-\{0\}$. 

$\bullet $ Let $\Sigma_p$ be the set of the permutations on $\{1,....,p\}$.

$\bullet$ For any ordered multiset $J=(j_1, \dots , j_r)$,  a {\bf multiset permutation}  $\sigma(J)$ of $J$  is an ordered arrangement of  the multiset $\{j_1, \dots , j_r\}$ where each element appears as often as it does in $J$.
% Even if $\sigma $ is not a function, we write it   $$ (\sigma(j_1), \dots , \sigma(j_r)).$$

$\bullet$ Let $M(p \times q, \R) $ denote the set of the $p \times q $  matrices
with entries in $\R$. For any $A
\in M(p \times q, \R)$, let $rk(A) $ denote the rank of $A$
and let $A^{(j)}$ be the $j$-th column of $A$.

$\bullet $
For any vector space $V$ over a field $K$  and any $v_1,\dots, v_k \in V$, let $\langle v_1,\dots, v_k\rangle $ be the 
span of $v_1,\dots,v_k$.

$\bullet $
 Let   $\mu =( [m_{i,j}, M_{i,j}])_{i,j}$ be a $p \times q $  interval matrix, where
  $m_{i,j} \leq M_{i,j}$  for any $i$ and $j$.  
 As we have already said, 
 given a matrix $A \in M(p \times q, \R) $, we say that  $ A \in \mu $ if and only if  $a_{i,j} \in [m_{i,j}, M_{i,j}] $ for any $i,j$.
  We define
$$mrk(\mu) = min\{rk(A)   | \; A \in \mu     \},$$
$$Mrk(\mu) = max\{rk(A) | \; A \in \mu     \}.$$
We call them respectively {\bf minimal rank} and {\bf maximal rank} of $\mu$.
Moreover, we define 
$$rkRange (\mu) =  \{ rk(A)  | \; A \in \mu \};$$ 
we call it the {\bf rank range} of $\mu$.

We say that the  entry $i,j$ of $\mu$ is a {\bf constant}
 if $ m_{i,j}= M_{i,j}$.

\begin{rem}
Let $\mu$ be an interval matrix. Observe that  $$rkRange(\mu) =
  [mrk(\mu)  , Mrk(\mu)] \cap \N.$$
\end{rem}
\begin{proof}
Obviously $rkRange(\mu) \subseteq
  [mrk(\mu)  , Mrk(\mu)] \cap \N
  $. We have to show the other inclusion.

There exist $A , B \in \mu $ with $rk(A)=
mrk(\mu) $ and $rk(B) =Mrk(\mu)$; let $(i_1,j_1),\dots, (i_k,j_k)$
be the entries of $A$ different from the corresponding entries of $B$ and, for every $r \in \{1,\dots,k\}$, let   
$C_r$ be the matrix  obtained from $A$ by changing the entries
 $(i_1,j_1),\dots, (i_r,j_r)$
 of $A$ into the corresponding entries of $B$. Obviously $C_k=B$, $C_1,\dots, C_k \in \mu$  and, for every $i$,
 the absolute value of the difference between the rank of 
 $C_i$ and the rank of $C_{i+1}$ is less than or equal to $1$.
 
 Therefore the set $\{rk(A), rk(C_1), \dots, rk(C_{k-1}), rk(B)\} $ is an interval of $\N$ and thus it must contain  
 $[rk(A), rk(B)] \cap \N$, that is 
 $ [mrk(\mu)  , Mrk(\mu)] \cap \N$; moreover,
 it is contained in $rkRange(\mu)$ (since 
   $A,C_1,\dots, C_{k-1}, B \in \mu$), so we can conclude. 
\end{proof}

\begin{defin}
We define the sum and the product of two intervals in $\R$  as follows: let $[e,f]$ and $[g,h]$ two intervals in $\R$ with $e \leq f$, $g \leq h$; we define
$$ [e,f] + [g,h] := [e+g, f+h],$$
$$ [e,f] \cdot [g,h] := [\min\{eg, eh, fg, fh\}, \max\{eg, eh, fg, fh\}].$$
In particular, for any $e \in \R$, $$ e \cdot [g,h] := [\min\{eg, eh\}, \max\{eg, eh\}].$$
\end{defin}
It is easy to see that $$\{x+y\; |\; x \in [e,f], \; y \in [g,h]\} \;= \;[e,f] + [g,h] ,$$ 
$$\{xy\; |\; x \in [e,f], \; y \in [g,h]\} \;= \;[e,f] \cdot [g,h] .$$  
Observe that in general, $([a,b]+[c,d])-[c,d] \neq [a,b]$.
\begin{defin}
Let $\mu $ be an interval matrix. We say that another interval matrix $\mu'$ is obtained from $\mu $ by  an  {\bf elementary row operation} if it is obtained from $\mu$ 
by one of the following operations:

I) interchanging two rows,

II) multiplying a row by a nonzero real number,

III)  adding to a row the multiple of another row by a real number.

In an analogous way we may define  {\bf elementary column operations}.
\end{defin}

\begin{rem} \label{op}
Let $\mu$ and $\mu'$ be two interval matrices such that 
$\mu'$ is obtained from $\mu$ by elementary row (or column) operations. Then, obviously, 
\begin{equation} \label{elop}
rkRange(\mu) \subseteq  rkRange (\mu').
\end{equation}
Moreover, if $\mu'$ is obtained from $\mu$ only by elementary row (or column) operations of kind I or II, we have  the equality in (\ref{elop}).
\end{rem}

\section{Interval matrices containing rank-zero or rank-one matrices}

\begin{rem} \label{mrk0}
Let  
   $\mu =( [m_{i,j}, M_{i,j}])_{i,j}$,  with $m_{i,j} \leq M_{i,j}$
 for any  $i \in \{1,\dots,p\}$  and $ j \in \{1,\dots,q\}$, be an interval matrix. 
 Obviously $mrk(\mu)=0$ if and only if,    for any  $i \in \{1,\dots,p\}$  and $ j \in \{1,\dots,q\}$, the interval 
 $[m_{i,j}, M_{i,j}]$ contains $0$.
 
 Moreover, 
 if $\mu'$ is the interval matrix obtained from 
 $\mu$ by deleting the columns and the rows such that all
 their entries contain $0$, we have that  $mrk(\mu)= mrk(\mu')$.

\end{rem}

\begin{defin} We say that an interval matrix $\mu $ 
is {\bf reduced} if every column and every row has at least one entry not containing $0$.
Reducing an interval matrix means eliminating every  row and every column 
such that each of its entries contains $0$. 
\end{defin}

\begin{rem} \label{po}
 Let $\mu$ be a $p \times q$ interval matrix.   Let $i \in \{1,\dots,p\}$ and $ j \in \{1,\dots,q\}$
 and
 let $\mu_{i,j}=[a,b]$ with $a \leq 0 \leq b$.
Define  $ \mu'$ and $ \mu''$  to be the  interval matrices  such that $$ \mu'_{i,j}=[a,0], \;\;\;\;\;  \;\;\; \mu''_{i,j}=[0,b]$$ and 
$\mu'_{t,s}= \mu''_{t,s} = \mu_{t,s}$ for any $(t,s) \neq (i,j)$.
 Let $r \in \N$. Then obviously there exists $A \in \mu $ with $rk(A) =r$ 
if and only if  either 
there exists $A \in \mu' $ with $rk(A) =r$ or 
there exists $A \in \mu'' $ with $rk(A) =r$.

In particular, to study whether an interval matrix
  $\mu $
contains a rank-$r$ matrix, it is sufficient to consider the case where,  for any $i,j$, either $\mu_{i,j} \subseteq \R_{\geq 0}$ or $\mu_{i,j} \subseteq \R_{\leq 0}$. 

 Moreover, by Remark
\ref{op}, we can suppose $\mu_{i,j} \subseteq \R_{\geq 0}$ for every $(i,j) $ 
such that  either $i$ or $j$ is equal to $1$.

\end{rem}

\begin{rem} \label{tronco} Let $\mu$ be 
 a reduced interval matrix such that 
$\mu_{i,j} \subseteq \R_{\geq 0}$
 for every $(i,j) $ with either $i$ or $j$ is equal to $1$.
 Obviously, if there exists $(i,j)$ such that $\mu_{i,j} \subseteq
\R_{< 0} $, then $\mu $ does not contain a rank-one matrix. 
Otherwise, that is $M_{i,j} \geq 0$ for any $i,j$,
 define $\overline{\mu}$ to be the interval matrix  such that 
 $$\overline{\mu}_{i,j} =
[ \max\{0, m_{i,j}\}, M_{i,j}]$$ 
 for any $i,j$. Obviously 
$\mu $ contains a rank-one matrix if and only if 
$\overline{\mu}$ contains a rank-one matrix.
\end{rem}

So to study when a reduced interval matrix contains a rank-one matrix it is sufficient to study the problem for a reduced interval matrix $ \mu $,  
with $\mu_{i,j} \subseteq \R_{\geq 0}$
 for every $i,j $. The following theorem 
gives an answer for such a problem  in this case.

\begin{thm} \label{casopositivo}
 Let   $\mu =( [m_{i,j}, M_{i,j}])_{i,j}$ be a   $p \times q $ reduced  interval matrix  with $p,q \geq 2$ and  $0  \leq m_{i,j} \leq M_{i,j}$
 for any  $i \in \{1,\dots,p\}$  and $ j \in \{1,\dots,q\}$. There exists 
 $A \in \mu$ with $rk(A)=1$ if and only if, 
 for any $h \in \N$
 with $2 \leq h \leq 2^{\min\{p,q\}-1}$, for any
  $i_1,\dots, i_h \in \{1,\dots, p\}$, for any
$j_1,\dots, j_h \in \{1,\dots, q\}$ and for any $\sigma \in \Sigma_h$, we have:
\begin{equation} \label{assthm}
m_{i_1, j_1 }\dots m_{i_h, j_h} \leq M_{i_1, j_{\sigma(1)}}  \dots
 M_{i_h, j_{\sigma(h)}}.\end{equation}
\end{thm}

\begin{proof}
We can suppose $q \leq p$.

$\Longrightarrow$ Let 
 $A \in \mu$ with $rk(A)=1$. 
For any   $i_1,\dots, i_h \in \{1,\dots, p\}$, for any
  $j_1,\dots, j_h \in \{1,\dots, q\}$ and for any $\sigma \in \Sigma_h$, we have: 
$$m_{i_1, j_1 }\dots m_{i_h, j_h} \leq
a_{i_1, j_1 }\dots a_{i_h, j_h}= 
a_{i_1, j_{\sigma(1)}}  \dots
 a_{i_h, j_{\sigma(h)}}
\leq
M_{i_1, j_{\sigma(1)}}  \dots
 M_{i_h, j_{\sigma(h)}}.$$

$\Longleftarrow$
Observe that, since $\mu $ is reduced, there exists  
 $A \in \mu$ with $rk(A)=1$ if and only if there exist
 $x_1,\dots, x_p, c_2, \dots, c_q$ in $\R$ such that 
 $$ x_i \in [m_{i,1}, M_{i,1}] , \;\; c_j  x_i \in [m_{i,j}, M_{i,j}]
$$ for any $i,j$.
Since $\mu$ is reduced and $m_{i,j} \geq 0$, this is equivalent to ask that there exist
 $x_1,\dots, x_p, c_2, \dots, c_q$ in $\R_{>0}$ such that, for any $i,j$, 
 $$ x_i \in [m_{i,1}, M_{i,1}] , \;\; c_j  x_i \in [m_{i,j}, M_{i,j}].
$$ 

If we define $c_1 =1$, this is equivalent to the following condition:
 there exist
 $x_1,\dots, x_p, c_2, \dots, c_q$ in $\R_{>0}$ such that, 
 for any $i,j$,
 $$ x_i \in  \cap_{j \in \{2,....,q\} } \left[  \frac{m_{i,j}}{c_j}, \frac{M_{i,j}}{c_j}  \right].
$$ 

Obviously this is equivalent to say that there exist
 $ c_2, \dots, c_q$ in $\R_{>0}$ such that 
 \begin{equation} \label{bo1}
 \frac{m_{i,t}}{c_t} \leq  \frac{M_{i,k}}{c_k}  
\end{equation}
for any $i    \in \{1, \dots, p\}$, $t,k \in \{1, \dots, q\}$. 
We will use the following notation:
$$m_{\left( \begin{array}{c} i_1 , \dots , i_s\\ j_1 , \dots , j_s  \end{array} \right)}:= \; m_{i_1, j_1} \cdot \dots \cdot m_{i_s,j_s},
\hspace*{1.2cm} M_{ \left(\begin{array}{c} i_1 , \dots , i_s\\ j_1 , \dots , j_s  \end{array}\right)}:= \; M_{i_1, j_1} \cdot \dots \cdot M_{i_s,j_s}.$$

Let us prove, by induction on $r$, that, for any $r \in \{2,\dots,q\} $, there exist $c_2,\dots, c_r 
\in \R_{>0}$  such that, for any $k,t \in \{1,\dots, r\}$, $b \in [1, 2^{q-r}] \cap \N$,
 $i_1,\dots,i_b \in\{1,\dots,p\}$, $ j_1,\dots,j_{b-1} \in\{1,\dots,q\}$, $ \sigma$ multiset permutation of $( j_1, \dots, j_{b-1}, k)$, 
 we have:
\begin{equation} \label{bo2}
m_{ \left(\begin{array}{c} i_1, \dots , i_{b-1} , i_b \\
j_1 ,\dots,  j_{b-1},  t
\end{array}\right)} \;
 c_k  \; \; \leq \;\;
 M_{ \left(\begin{array}{c} i_1, \dots,  i_{b-1 } , i_b\\
\sigma(j_1, \dots, j_{b-1}, k) 
\end{array} \right)} \;
 c_t ,
\end{equation}
where $c_1 $ is defined to be $1$.
Observe that the condition above  implies $(\ref{bo1})$ (take $b=1$ and $r=q$).

As to the induction \underline{base case $r=2$}, we have to prove that there exists $c_2 \in \R_{>0}$ such that  
for any  $b, d \in  [1, 2^{q-2}] \cap \N$,
 $i_1,\dots,i_b , u_1,\dots,u_d \in\{1,\dots,p\}$, $ j_1,\dots,j_{b-1} , v_1,\dots,v_{d-1} \in \{1,\dots,q\}$,  
 $ \sigma$ multiset permutation of $(j_1, \dots, j_{b-1}, 1) $,   $ \gamma$ multiset  permutation of $( v_1, \dots, v_{d-1}, 2) $,  
\begin{equation} \label{base1}
m_{ \left(\begin{array}{c} i_1 , \dots , i_{b-1} , i_b \\
j_1, \dots , j_{b-1} , 2
\end{array}\right)} \;  \;\leq \;\; M_{ \left(\begin{array}{c} i_1 , \dots , i_{b-1 } , i_b\\
\sigma(j_1, \dots , j_{b-1},   \ 1) \end{array} \right)} 
  \; c_2
  \end{equation}
and 
\begin{equation} \label{base2}
   m_{ \left(\begin{array}{c} u_1 , \dots, u_{d-1 } , u_d\\
v_1,  \dots , v_{d-1} , 1
\end{array} \right)} 
  \; c_2 
\; \; \leq \;\; M_{ \left(\begin{array}{c} u_1 , \dots , u_{d-1} , u_d \\
\gamma(v_1, \dots  , v_{d-1}  , 2) 
\end{array}\right)} .
  \end{equation}
  Observe that, if  $$m_{ \left(\begin{array}{c} i_1 , \dots , i_{b-1} , i_b \\
j_1, \dots , j_{b-1} , 2
\end{array}\right)}=0,$$ then 
the inequality $(\ref{base1})$ is implied 
by the condition $c_2 > 0$ and, if 
$$ m_{ \left(\begin{array}{c} u_1 , \dots , u_{d-1 } , u_d\\
v_1,  \dots, v_{d-1} , 1
\end{array} \right)} =0,$$ then the inequality 
  $(\ref{base2})$ is always true.
  So we can consider only $b, d \in  [1, 2^{q-2}] \cap \N$,
 $i_1,\dots,i_b , u_1,\dots,u_d \in\{1,\dots,p\}$, $ j_1,\dots,j_{b-1} , v_1,\dots,v_{d-1} \in \{1,\dots,q\}$ such that the terms $m_{..}$ in (\ref{base1}), (\ref{base2}) are positive.
 So a $c_2$ as we search for exists if and only if 
for any  $b, d \in   [1, 2^{q-2}] \cap \N$,
 $i_1,\dots,i_b , u_1,\dots,u_d \in\{1,\dots,p\}$, $ j_1,\dots,j_{b-1} , v_1,\dots,v_{d-1} \in \{1,\dots,q\}$ such that $$m_{ \left(\begin{array}{c} i_1 , \dots , i_{b-1} , i_b \\
j_1 ,\dots,  j_{b-1} , 2
\end{array}\right)} \hspace*{1cm} \mbox{\rm and} \hspace*{1cm}  m_{ \left(\begin{array}{c} u_1 , \dots , u_{d-1 } , u_d\\
v_1,  \dots , v_{d-1} , 1
\end{array} \right)} $$ are positive,  
 $ \sigma$ multiset permutation of $( j_1, \dots, j_{b-1}, 1) $,   $ \gamma$ multiset permutation of $(v_1, \dots, v_{d-1}, 2)$,  we have that
   $$
  m_{ \left(\begin{array}{c} u_1 , \dots , u_{d-1 } , u_d , i_1 , \dots , i_{b-1} , i_b\\
v_1,  \dots , v_{d-1} , 1 , j_1, \dots , j_{b-1} , 2
\end{array} \right)}  \;\leq \;\; M_{ \left(\begin{array}{c} u_1 , \dots , u_{d-1} , u_d ,  i_1 , \dots , i_{b-1 } , i_b \\
 \gamma(v_1, \dots , v_{d-1}, 2)\, \sigma(j_1, \dots , j_{b-1}, 1)   \end{array} \right)} 
$$
%\hspace{-0.3cm} 
and this follows from our assumption 
(\ref{assthm}).

Let us prove \underline{the induction step}.
By induction assumption we can suppose there exist  $c_2,\dots, c_{r-1} \in \R_{>0} $ such that
for any $k,t \in \{1,\dots, r-1\}$, $b \in [1, 2^{q-r+1}] \cap \N$,
 $i_1,\dots,i_b \in\{1,\dots,p\}$, $ j_1,\dots,j_{b-1} \in\{1,\dots,q\}$, $ \sigma$ multiset permutation of $( j_1, \dots, j_{b-1}, k)$, the inequality
 $(\ref{bo2})$ holds. We want to show that there exists $c_r \in \R_{>0}$ such that for any $t,l \in \{1,\dots, r-1\}$, $b, d 
\in [1, 2^{q-r}] \cap \N$,
 $i_1,\dots,i_b, u_1,\dots,u_d \in\{1,\dots,p\}$, $ j_1,\dots,j_{b-1}, v_1,\dots,v_{d-1} \in\{1,\dots,q\}$, 
 $ \sigma$ multiset  permutation of $\{ j_1, \dots, j_{b-1}, r\} $,   $ \gamma$ multiset permutation of
  $\{ v_1,\dots, v_{d-1}, l\} $,  
\begin{equation} \label{rt}
m_{ \left(\begin{array}{c} i_1 , \dots , i_{b-1} , i_b \\
j_1, \dots , j_{b-1} , t
\end{array}\right)} \;  c_r \;\; \leq \;\; M_{ \left(\begin{array}{c} i_1 , \dots , i_{b-1 } , i_b\\
\sigma(j_1, \dots,  j_{b-1}, r) \end{array} \right)} 
  \; c_t
  \end{equation}
and 
\begin{equation} \label{rl}
 M_{ \left(\begin{array}{c} u_1 , \dots , u_{d-1} , u_d \\
\gamma(v_1, \dots,  v_{d-1}, l) 
\end{array}\right)} \;
c_r \;\; \geq \;\; m_{ \left(\begin{array}{c} u_1 , \dots , u_{d-1 }, u_d\\
v_1,  \dots,  v_{d-1} , r
\end{array} \right)} 
  \; c_l .\end{equation}
Observe that when $$m_{ \left(\begin{array}{c} i_1 , \dots , i_{b-1} , i_b \\
j_1, \dots , j_{b-1} , t
\end{array}\right)} =0,$$
the inequality
 (\ref{rt}) is always verified 
   and,  if $$m_{ \left(\begin{array}{c} u_1 ,\dots , u_{d-1 } , u_d\\
v_1,  \dots , v_{d-1}  , r
\end{array} \right)} =0,$$ then  
   the inequality (\ref{rl}) is implied 
by the condition $c_r >0$. So we can consider only
 $t \in \{1,\dots, r-1\}$, $b, d \in [1, 2^{q-r}] \cap \N$,
 $i_1,\dots,i_b, u_1,\dots,u_d \in\{1,\dots,p\}$, $ j_1,\dots,j_{b-1}, v_1,\dots,v_{d-1} \in\{1,\dots,q\}$ such that 
the terms $m_{..}$ appearing   in (\ref{rt}) and
in (\ref{rl}) are positive.
Thus 
 a $c_r$  as we search for exists if and only if 
 for any $t,l \in \{1,\dots, r-1\}$, $b, d \in   [1, 2^{q-r}] \cap \N$,
 $i_1,\dots,i_b, u_1,\dots,u_d \in\{1,\dots,p\}$, $ j_1,\dots,j_{b-1}, v_1,\dots,v_{d-1}  \in\{1,\dots,q\}$ such that 
 $$m_{ \left(\begin{array}{c} i_1 , \dots , i_{b-1} , i_b \\
j_1, \dots , j_{b-1} , t
\end{array}\right)}  \hspace*{1cm}  \mbox{\rm and}  \hspace*{1cm}   m_{ \left(\begin{array}{c} u_1 ,\dots , u_{d-1 } , u_d\\
v_1,  \dots , v_{d-1} , r
\end{array} \right)}$$ are positive, 
 $ \sigma$ multiset permutation of $( j_1, \dots, j_{b-1}, r ) $,   $ \gamma$ multiset  permutation of $( v_1,\dots v_{d-1}, l) $,  we have:
 $$ 
m_{ \left(\begin{array}{c} i_1 , \dots , i_{b-1 } , i_b , u_1 , \dots , u_{d-1} , u_d \\
j_1, \dots ,j_{b-1} ,   t
, v_1 , \dots ,v_{d-1} ,   r  \end{array} \right)}
 \; c_l  \,\;
  \leq  
  M_{ \left(\begin{array}{c} i_1 , \dots , i_{b-1 } , i_b , u_1 , \dots , u_{d-1} , u_d \\
\sigma(j_1 \dots , j_{b-1}, r)
\, \gamma(v_1, \dots , v_{d-1},  l)  \end{array} \right)}
 \; c_t  
    $$
    and this is  true by induction assumption.
\end{proof}

{\bf Example.}
Let $$ \mu = 
\left( \begin{array}{cccc}
[2,3] & [1,6]   & [-2, 2]  & [-3,-1]\\ 
\, [1,2] & [2,3] & [-2,3] & [ -2,3] \\ 
 \,[1,4] & [0,2]   &  [3,4] & [-1,0]
\end{array}
\right).$$ 
By Remarks  \ref{op} and \ref{po}, the interval matrix  $\mu$ contains a rank-one matrix if and only if at least one of the following interval  matrices contains a rank-one matrix  $$\mu':= 
\left( \begin{array}{cccc}
[2,3] & [1,6]   & [0, 2]  & [1,3]\\ 
\, [1,2] & [2,3] & [-2,3] & [ -3,2] \\ 
 \,[1,4] & [0,2]   &  [3,4] & [0,1]
\end{array}
\right),
\;\;\;\;\; \mu'':=
\left( \begin{array}{cccc}
[2,3] & [1,6]   & [-2, 0]  & [1,3]\\ 
\, [1,2] & [2,3] & [-2,3] & [ -3,2] \\ 
 \,[1,4] & [0,2]   &  [3,4] & [0,1]
\end{array}
\right)
$$
and, by Remark \ref{op}, this is equivalent to ask that at least one of the interval matrices 
$\mu'$ and $\mu'''$
contains a rank-one matrix, where 
 $$ \mu''':=
\left( \begin{array}{cccc}
[2,3] & [1,6]   & [0, 2]  & [1,3]\\ 
\, [1,2] & [2,3] & [-3,2] & [ -3,2] \\ 
 \,[1,4] & [0,2]   &  [-4,-3] & [0,1]
\end{array}
\right).$$   By Remark \ref{tronco}, the interval matrix 
$\mu'''$ does not contain a rank-one matrix since $\mu'''_{i,j} \subset \R_{\geq 0}$ for any $i,j $ such that either $i$ or $j$ is equal to $1$, but $\mu'''_{3,3}
\subset \R_{< 0}$.  Hence $ \mu$  contains a rank-one matrix if and only if $\mu'$ contains a 
rank-one matrix and, by Remark \ref{tronco}, this is equivalent  
to ask  that  $\overline{\mu'}$
 contains a rank-one matrix, where  $$\overline{\mu'}:= 
\left( \begin{array}{cccc}
[2,3] & [1,6]   & [0, 2]  & [1,3]\\ 
\, [1,2] & [2,3] & [0,3] & [ 0,2] \\ 
 \,[1,4] & [0,2]   &  [3,4] & [0,1]
\end{array}
\right).$$
By Theorem \ref{casopositivo}, the interval matrix 
$\overline{\mu'}$ does not contain a rank-one matrix because for instance 
the product of the mimina of the intervals 
$\overline{\mu'}_{1,1}$, $\overline{\mu'}_{2,2}$ and
$\overline{\mu'}_{3,3}$ is greater than the 
 product of the maxima of the intervals 
$\overline{\mu'}_{2,1}$, $\overline{\mu'}_{3,2}$ and
$\overline{\mu'}_{1,3}$.
So we can conclude that  $\mu $ does not contain a rank-one matrix.
\begin{rem}
It is not true that, given a   $p \times q $ interval matrix,  
   $\mu =( [m_{i,j}, M_{i,j}])_{i,j}$,  with $m_{i,j} \leq M_{i,j}$
 for any  $i \in \{1,\dots,p\}$  and $ j \in \{1,\dots,q\}$, then 
  there exists 
 $A \in \mu$ with $rk(A)=1$ if and only if, for any $h \in [2 , 2^{\min\{p,q\}-1}] \cap \N$, for any
 $i_1,\dots, i_h \in \{1,\dots, p\}$, for any
  $j_1,\dots, j_h \in \{1,\dots, q\}$ and for any $\sigma \in \Sigma_h$, we have 
\begin{equation} \label{bo3}
\mu_{i_1, j_1 } \cdot \dots \cdot \mu_{i_h, j_h} \cap \mu_{i_1,j_{\sigma(1)}} \cdot \dots \cdot
 \mu_{i_h, j_{\sigma(h)}} \neq \emptyset.
 \end{equation}
 In fact, for instance the interval matrix 
 $$\left( \begin{array}{ccc} \left[-3, \frac{1}{3}\right] & [2,4] & 
 [0,1] \\ 
 1 & [-1 ,3 ] & 1
 \end{array}
 \right)
  $$ 
  satisfies condition (\ref{bo3}), but it is easy to see that it does not contain any rank-one matrix.
\end{rem}

\section{Maximal rank of matrices contained in an interval matrix}

\begin{defin}
Given a $p \times p $ interval  matrix, $\nu$, a {\bf partial generalized
diagonal} ({\bf pg-diagonal} for short) of length $k$ of $\nu$ is a $k$-uple of the kind $$
(\nu_{i_1, j_1},\dots, \nu_{i_k, j_k})$$ 
for some  $\{i_1, \dots i_k\}$ and $ \{j_1, \dots, j_k\} $ subsets of $ \{1,\dots ,p\}$.

Its {\bf complementary matrix} is defined to be the submatrix of $\mu$ given by the rows and columns whose indices are respectively in  $\{1,\ldots , p\}- \{i_1, \ldots , i_k\}$ and 
in  $\{1,\ldots , p\}- \{j_1, \ldots , j_k\}$. 

We say that a pg-diagonal is {\bf totally nonconstant} if and only if all its entries are not constant.

We define $det^c(\mu) $ as follows:
$$ det^c(\mu) = 0+ \sum_{\sigma \in \Sigma_p \; s.t. \; \mu_{1, \sigma(1)}, \dots, \mu_{p, \sigma(p)} \;  are \; constant} \epsilon (\sigma) \,
\mu_{1, \sigma(1)} \cdot \ldots  \cdot\mu_{p, \sigma(p)} ,$$ 
where $\epsilon (\sigma) $ is the sign of the permutation $\sigma$.

For every pg-diagonal of length $p$, say 
$
\mu_{1, \sigma(1)} \cdot \ldots  \cdot\mu_{p, \sigma(p)} $
for some $\sigma \in \Sigma_p$, 
 we call  $\epsilon (\sigma) $ also  the sign of the pg-diagonal.
\end{defin}

\begin{thm} \label{Mrkquadrate}
Let $\mu $ be a $p \times p $ interval matrix. Then $Mrk(\mu) < p$ if and only if the following conditions hold:

(1) in $\mu$ there is no totally nonconstant pg-diagonal of length $p$,

(2) the complementary matrix of every  totally nonconstant pg-diagonal of  length between $0$ and $p-1$ has $det^c $ equal to $0$.

\end{thm}

\begin{proof}
$\Longrightarrow$ 
We prove the statement by induction on $p$. For $p=1 $ the statement is obvious. 
 Suppose $p \geq 2$ and that  the statement is true for $(p-1) \times (p-1) $ interval matrices. Let $\mu$ be a $(p\times p )$ interval matrix such that 
 $Mrk(\mu) <p$, that is $det(A)=0$ for every 
 $A \in \mu$. 
 
If $\mu$ contained a  totally nonconstant  pg-diagonal
of length $p$, say $\mu_{i_1,j_1}, \ldots
, \mu_{i_p, j_p}$, then  $\mu_{\hat{i_1}, \hat{j_1}}$ has obviously a totally nonconstant pg-diagonal of length $p-1$, so by induction assumption,  $Mrk(\mu_{\hat{i_1}, \hat{j_1}})=p-1$. Thus there exists $B \in 
\mu_{\hat{i_1}, \hat{j_1}} $ with $det(B) \neq 0$. So, for any choice of elements $x_{i_1, j}\in \mu_{i_1, j}$  for $j \neq j_1$ and
 $x_{i, j_1} \in \mu_{i, j_1}$  for $i \neq i_1$, 
we can find $x \in \mu_{i_1,j_1}$ such that the determinant of the matrix $X$ defined by $X_{\hat{i_1}, \hat{j_1}}=B$, $X_{i_1,j_1}=x$, $X_{i,j_1} = x_{i, j_1}$
for any $i \neq i_1$ and $X_{i_1,j} = x_{i_1, j}$
for any $j \neq j_1$  is nonzero, which is absurd. 
So we have proved that (1) holds. 

Moreover, by contradiction, 
suppose (2) does not hold. 
Thus in $\mu$ there exists a totally nonconstant pg-diagonal 
of length $k$ with $ 0 \leq k \leq p-1$  whose complementary matrix 
has $det^c$ nonzero.
If there exists such a diagonal with $k \geq 1$, say 
$\mu_{i_1,j_1}, \ldots
, \mu_{i_k, j_k}$, then also $ \mu_{\hat{i_1},\hat{j_1}}$ does not satisfy (2), so, by induction assumption,   there exists $B \in 
\mu_{\hat{i_1}, \hat{j_1}} $ with $det(B) \neq 0$ and we conclude as before.
On the other hand, suppose that $det^c(\mu) \neq 0$ 
but the complementary matrix of every 
totally nonconstant pg-diagonal 
of length $k$ with $ 1 \leq k \leq p-1$ 
has $det^c$ equal to zero; we call this assumption ($\ast$).

Let $A \in \mu$. 
By  (1), we can write $det (A)$ as the sum of: 

- the  sum (with sign) 
of the product of the entries of the  pg-diagonals of $A$ of length $p$ such the corresponding entries of $\mu$ are all constant, 

- the  sum (with sign) 
of the product of the entries of the  pg-diagonals of $A$  of length $p$ such all the corresponding entries of $\mu$  apart from one are constant, 

....

- the  sum (with sign) 
of the product of the entries of the  pg-diagonals of $A$  of length $p$ such all the corresponding entries of $\mu$  apart from $p-1$  are constant. 

The first sum coincides with  $det^c(\mu) $, so it is nonzero by the assumption ($\ast$); we can 
write the second sum by collecting the terms containing the same entry corresponding to the nonconstant  entry of 
$\mu$; so, by assumption ($\ast$), we get that 
this sum is  zero; we argue analogously for the other sums. So we can conclude that $det(A)$ is nonzero, which is absurd.

$\Longleftarrow $ Let $\mu$ be a  matrix 
satisfying (1) and (2) and let $A \in \mu$. 
By assumption (1), we can write $det (A)$ as the sum of: 

- the  sum (with sign) 
of the product of the entries of the  pg-diagonals of $A$ of length $p$ such the corresponding entries of $\mu$ are all constant, 

- the  sum (with sign) 
of the product of the entries of the  pg-diagonals of $A$  of length $p$ such all the corresponding entries of $\mu$  apart from one are constant, 

....

- the  sum (with sign) 
of the product of the entries of the  pg-diagonals of $A$  of length $p$ such all the corresponding entries of $\mu$  apart from $p-1$  are constant. 

The first sum is zero by assumption; we can 
write the second sum by collecting the terms containing the same entry corresponding to the nonconstant  entry of 
$\mu$; so by assumption we get that also 
this sum is  zero. We argue analogously for the other sums. 
\end{proof}

\begin{cor} \label{Mrk}
Let $\mu $ be an interval matrix. Then $Mrk(\mu)$
is the maximum of the natural numbers $t$ such that there is a $ t \times t $ submatrix  of
 $\mu$ either with a totally nonconstant pg-diagonal  of  length  $t$ or with a totally nonconstant
 pg-diagonal of  length between $0$ and $t-1$ whose complementary matrix has $det^c  \neq 0$.
 \end{cor}

\begin{proof}
Let $\overline{t}$ be  the maximum of the natural numbers $t$ such that there is a $ t \times t $ submatrix  of
 $\mu$ with $Mrk$ equal to $t$ and
 let $\tilde{t}= Mrk(\mu)$.
 
  Obviously  $ \tilde{t} \geq \overline{t}$.
 On the other hand, let $A \in \mu $ such that $rk(A)= \tilde{t}$. Then there exists a $ \tilde{t} \times \tilde{t} $ submatrix of $A$ with nonzero determinant. Hence the corresponding submatrix 
 of $\mu$ is a  $ \tilde{t} \times \tilde{t} $ submatrix with $Mrk$ equal to $\tilde{t}$,
 so $ \overline{t} \geq \tilde{t}$. By Theorem \ref{Mrkquadrate} we can conclude.
\end{proof}

\section{The case of the matrices with three columns}

Let  
   $\mu =( [m_{i,j}, M_{i,j}])_{i,j}$,  with $m_{i,j} \leq M_{i,j}$
 for any  $i \in \{1,\dots,p\}$  and $ j \in \{1,2,3\}$, be an interval matrix. By Remarks \ref{mrk0},  \ref{po}, \ref{tronco}   and Theorem \ref{casopositivo} we can see easily whether $mrk(\mu) \in 
 \{0,1\}$ and by Corollary \ref{Mrk} we can calculate $Mrk(\mu) $. So if $mrk(\mu) \in 
 \{0,1\}$ we can calculate $rkRange(\mu)$. Moreover,
obviously, if $mrk(\mu) \not\in \{0,1\}$ and $Mrk(\mu) =2$,
we must have  $mrk(\mu)=2 $ and we  get that 
 every matrix contained in $\mu$ has rank $2$.
 So at the moment, we do not know 
  $rkRange(\mu)$
only if  $mrk(\mu) \not\in \{0,1\}$ and $Mrk(\mu) =3$. In this case, $mrk(\mu)$ can 
be only $2$ or $3$. 
In this section we give a criterion to  establish if,
 in this case, $mrk(\mu)$ is equal to $2 $ or to $3$.
 
\begin{lem} \label{imp}
(a) Let $a_i , c_j  \in \R_{\geq 0}$ and 
$ b_i , z_i  , d_j, u_j  \in \R$ for $i=1,\dots,k$ and
$j=1,\dots,h$ such that 
 and $a_{\overline{i}} $ and $c_{\overline{j}}$ are nonzero 
 for some $\overline{i} \in \{1,\dots,k\}$ and 
 $ \overline{j}  \in \{1,\dots, h\}$.
Then there exist $\lambda, \gamma \in \R_{\geq 0}$ such that, for every $i$ and $j$, 
 $$ \left\{ \begin{array}{l} 
 \lambda \, a_i + \gamma \, b_i \leq z_i ,\\
 \lambda \,  c_j + \gamma \, d_j \geq u_j
 \end{array}
  \right.$$ 
if and only if
there exist $\lambda, \gamma \in \R_{\geq 0}$ such that, for every $i$ and $j$, 
 $$ \left\{ \begin{array}{l} 
 \lambda  \, a_i  c_j + \gamma \, b_i c_j \leq z_i c_j ,\\
 \lambda \,  c_j  a_i+ \gamma\,  d_j a_i \geq u_j a_i .
 \end{array}
  \right.$$ 
(b)  Let $a_i  \in \R_{\geq 0}$  and $ 
 b_i , c_i , z_i , u_i \in \R$ for $i=1,\dots,k$; suppose 
  there exists $\overline{i}$ such that $a_{\overline{i}} \neq 0$.
  
Then there exist $\lambda, \gamma \in \R_{\geq 0}$ such that, for every $i$, 
 \begin{equation} \label{eqb}
  \left\{ \begin{array}{l} 
 \lambda \, a_i + \gamma \, b_i \leq z_i ,\\
 \lambda \,  a_i + \gamma \, c_i \geq u_i 
 \end{array}
  \right.
  \end{equation}
if and only if
there exists $ \gamma \in \R_{\geq 0}$ 
such that, for every $i,j \in \{1,\dots, k\}$, 
 $$a_j(  z_i - \gamma \, b_i )\geq  a_i (u_j - \gamma \, c_j ) $$ 
and $$ z_i - \gamma \, b_i \geq 0.$$
(c) Let $x_i, y_i \in \R$ for $i=1,\dots,k$. Then there exists $\gamma \in \R_{\geq 0}$  such that, for every $i$,
\begin{equation} \label{eqc}
 \gamma \, x_i \geq y_i
 \end{equation}
  if and only if, for every $i$,
$x_i \leq 0$ implies $y_i \leq 0$ and 
$$ y_j x_i \geq y_i x_j $$ for every $i,j$ such that 
$x_j > 0$ and $x_i < 0$. 
\end{lem}

\begin{proof}
Part (a) is obvious.
Let us prove (b).
There exist $\lambda, \gamma \in \R_{\geq 0}$ such that, for every $i$, 
  (\ref{eqb}) holds if and only if 
there exist $\lambda, \gamma \in \R_{\geq 0}$ such that
$$  \left\{ \begin{array}{l} 
0 \leq z_i - \gamma\, b_i \hspace*{1cm} \forall i 
\;\; \mbox{s.t. } a_i=0,\\
0 \geq u_i - \gamma \, c_i \hspace*{1cm} \forall i 
\;\; \mbox{s.t. } a_i=0,\\
 \lambda  \leq (z_i -\gamma \, b_i )/a_i \hspace*{1cm} \forall i 
\;\; \mbox{s.t. } a_i \neq 0, \\
 \lambda  \geq (u_i -\gamma \, c_i )/a_i \hspace*{1cm} \forall i \;\; \mbox{s.t. } a_i \neq 0, 
 \end{array}
  \right.$$
  and this is equivalent to the existence of $
\gamma \in \R_{\geq 0}$ such that
$$  \left\{ \begin{array}{ll} 
0 \leq z_i - \gamma\, b_i & \forall i 
\;\; \mbox{s.t. } a_i=0,\\
0 \geq u_i - \gamma \, c_i & \forall i 
\;\; \mbox{s.t. } a_i=0,\\
 (u_j -\gamma \, c_j )/a_j  \leq (z_i -\gamma \, b_i )/a_i & \forall i,j 
\;\; \mbox{s.t. } a_i a_j\neq 0, \\
z_i -\gamma \, b_i  \geq 0 & \forall i \;\; \mbox{s.t. } a_i \neq 0 ,
 \end{array}
  \right.$$
and we can easily prove that this is equivalent to 
the existence of  $ \gamma \in \R_{\geq 0}$ 
such that, for every $i,j$, we have that 
 $a_j(  z_i - \gamma \, b_i )\geq  a_i (u_j - \gamma \, c_j ) $ 
and $ z_i - \gamma \, b_i \geq 0$.

Let us prove (c). There exists $\gamma \in \R_{\geq 0}$  such that, for every $i$,
 (\ref{eqc}) holds if and only if 
 there exists $\gamma \in \R_{\geq 0}$  such that
 $$
   \left\{ \begin{array}{ll} 
\gamma  \geq y_i/x_i & \forall i 
\;\; \mbox{s.t. } x_i > 0,\\
\gamma  \leq y_i/x_i & \forall i 
\;\; \mbox{s.t. } x_i < 0,\\
y_i \leq 0  & \forall i 
\;\; \mbox{s.t. } x_i = 0,
 \end{array}
  \right.$$
  and this is equivalent to the following condition:
   for every $i$,
$x_i \leq 0$ implies $y_i \leq 0$ and 
$y_j x_i \geq y_i x_j $ for every $i,j$ such that 
$x_j > 0$ and $x_i < 0$. 
\end{proof}

\begin{cor} \label{corimp}
 Let $a_i , c_j  \in \R_{\geq 0}$ and 
$ b_i , z_i  , d_j, u_j  \in \R$ for $i, j \in \{1,\dots,k\}$  such that 
 and $a_{\overline{i}} $ and $c_{\overline{j}}$ are nonzero  for some $\overline{i} , \overline{j}\in \{1,\dots,k\}$.
Then there exist $\lambda, \gamma \in \R_{\geq 0}$ such that, for every $i$ and $j$, 
 \begin{equation} \label{eqcor} \left\{ \begin{array}{l} 
 \lambda \, a_i + \gamma \, b_i \leq z_i, \\
 \lambda \,  c_j + \gamma \, d_j \geq u_j  
 \end{array}
  \right.
  \end{equation}
if and only if,
for any $i,j,r,s $:

$	\bullet$ $b_i \geq 0 $ implies $z_i \geq 0$, 

$\bullet$ $b_i >0 $ and $b_j < 0 $
imply $ b_i z_j \geq b_j z_i$,

$\bullet$  $a_i   d_r -   b_i  c_r \leq 0 $ implies   $a_i  u_r   - c_r   z_i  \leq 0 $,

$ \bullet $ $a_i  d_r -    b_i  c_r <0 $
and $a_{j}   d_{s} -    b_{j}  c_{s} >0 $
imply $$(a_i   d_r -    b_i c_r)
(a_{j}   u_{s}    -  c_{s}   z_{j} )
 \geq
(a_{j}   d_{s} -    b_{j}  c_{s})
(a_i   u_r    -  c_r   z_i ) ,$$

$ \bullet $ $b_{j}< 0 $ and   $a_i   d_r -    b_i  c_r <0 $ imply 
$z_{j} (a_i    d_r -    b_i  c_r ) 
\leq b_{j}(a_i   u_r    -  c_r   z_i ) $,

$ \bullet $ $b_{j} > 0 $ and   $a_i    d_r -    b_i  c_r > 0 $ imply 
$z_{j} (a_i    d_r -    b_i  c_r ) 
\geq b_{j}(a_i  u_r    -  c_r   z_i ) $.
\end{cor}

\begin{proof}
By Remark \ref{imp} (a),  there exist $\lambda, \gamma \in \R_{\geq 0}$ such that, for every $i$ and $j$, (\ref{eqcor}) holds  if and only if 
there exist $\lambda, \gamma \in \R_{\geq 0}$ such that, for every $i$ and $j$, 
 $$ \left\{ \begin{array}{l} 
 \lambda  \, a_i  c_j + \gamma \, b_i c_j \leq z_i c_j ,\\
 \lambda \,  c_j  a_i+ \gamma\,  d_j a_i \geq u_j a_i .
 \end{array}
  \right.$$ 
By Remark \ref{imp} (b), this is true if and only if 
there exists $\gamma \in \R_{\geq 0}$ such that, for every $i,j,r,s$, 
$$ \left\{ \begin{array}{l} 
   \gamma \, b_i c_j \leq  z_i  c_j ,\\
  a_s c_r  (  z_i  c_j - \gamma \, b_i c_j )   \geq 
  a_i c_j (u_r a_s - \gamma \,d_r a_s) , 
 \end{array}
  \right.$$ 
  which is obviously true if and only if 
there exists $\gamma \in \R_{\geq 0}$ such that, for every $i,r$, 
$$ \left\{ \begin{array}{l} 
   \gamma \,(- b_i)  \geq -z_i, \\
 \gamma \,(  a_i    d_r-   b_i  c_r   )    \geq 
  a_i  u_r    -  c_r   z_i .
 \end{array}
  \right.$$ 
By Remark \ref{imp} (c), this is true if and only if the following conditions hold for any $i,j,r, s, $:

$	\bullet$ $b_i \geq 0 $ implies $z_i \geq 0$, 

$\bullet$ $b_i >0 $ and $b_j < 0 $
imply $ b_i z_j \geq b_j z_i$,

$\bullet$  $a_i   d_r -   b_i  c_r \leq 0 $ implies   $a_i  u_r   - c_r   z_i  \leq 0 $,

$ \bullet $ $a_i  d_r -    b_i  c_r <0 $
and $a_{j}   d_{s} -    b_{j}  c_{s} >0 $
imply $$(a_i   d_r -    b_i c_r)
(a_{j}   u_{s}    -  c_{s}   z_{j} )
 \geq
(a_{j}   d_{s} -    b_{j}  c_{s})
(a_i   u_r    -  c_r   z_i ), $$

$ \bullet $ $b_{j}< 0 $ and   $a_i   d_r -    b_i  c_r <0 $ imply 
$z_{j} (a_i    d_r -    b_i  c_r ) 
\leq b_{j}(a_i   u_r    -  c_r   z_i ) $,

$ \bullet $ $b_{j} > 0 $ and   $a_i    d_r -    b_i  c_r > 0 $ imply 
$z_{j} (a_i    d_r -    b_i  c_r ) 
\geq b_{j}(a_i  u_r    -  c_r   z_i ) $.
\end{proof}

\begin{thm}
Let  
   $\mu =( [m_{i,j}, M_{i,j}])_{i,j}$,  with $m_{i,j} \leq M_{i,j}$
 for any  $i \in \{1,\dots,p\}$  and $ j \in \{1,2,3\}$, be a reduced  interval matrix. Suppose $m_{i,1} \geq  0$ for every $i \in \{1,\dots,p\}$. 
  Then $mrk(\mu) $ is less than or equal to $2$ if and only if,
 for $(v,w) =(2,3)$ or for  $(v,w) =(3,2)$, we have that at least one of (1),(2),(3),(4) holds:

(1) 
for any $i,j,r, s \in \{1,\dots,p\} $, we have:

$	\bullet$ $m_{i,v} \geq 0 $ implies $M_{i,w} \geq 0$, 

$\bullet$ $m_{i,v} >0 $ and $m_{j,v} < 0 $
imply $ m_{i,v} M_{j,w} \geq m_{j,v} M_{i,w}$,

$\bullet$  $m_{i,1}   M_{r,v} -   m_{i,v}  M_{r,1} \leq 0 $ implies   $m_{i,1}  m_{r,w}   - M_{r,1}   M_{i,w}  \leq 0 $,

$ \bullet $ $m_{i,1}  M_{r,v} -    m_{i,v}  M_{r,1} <0 $
and $m_{j,1}   M_{s,v} -    m_{j,v}  M_{s,1} >0 $
imply $$(m_{i,1}   M_{r,v} -    m_{i,v} M_{r,1})
(m_{j,1}   u_{s}    -  M_{s,1}   M_{j,w} )
 \geq
(m_{j,1}   M_{s,v} -    m_{j,v}  M_{s,1})
(m_{i,1}   m_{r,w}    -  M_{r,1}   M_{i,w} ) ,$$

$ \bullet $ $m_{j,v}< 0 $ and   $m_{i,1}   M_{r,v} -    m_{i,v}  M_{r,1} <0 $ imply 
$$M_{j,w} (m_{i,1}    M_{r,v} -    m_{i,v}  M_{r,1} ) 
\leq m_{j,v}(m_{i,1}   m_{r,w}    -  M_{r,1}   M_{i,w} ) ,$$

$ \bullet $ $m_{j,v} > 0 $ and   $m_{i,1}    M_{r,v} -    m_{i,v}  M_{r,1} > 0 $ imply 
$$M_{j,w} (m_{i,1}    M_{r,v} -    m_{i,v}  M_{r,1} ) 
\geq m_{j,v}(m_{i,1}  m_{r,w}    -  M_{r,1}   M_{i,w} ).$$

(2) the same conditions as in (1)
with $-M_{\cdot,v }$ instead of $m_{\cdot, v}$  and vice versa hold;

(3)  
 the same conditions as in (1)
with $-M_{\cdot,1 }$ instead of $m_{\cdot, 1}$  and vice versa hold;

(4)   the same conditions as in (1) 
with $-M_{\cdot,1 }$ instead of $m_{\cdot, 1}$ and vice versa and
with $-M_{\cdot,v }$ instead of $m_{\cdot, v}$  and vice versa hold.
\end{thm}

\begin{proof}
Obviously   $mrk(\mu) \leq 2 $  if and only if there exists $A 
\in \mu $  and $c_1, c_2, c_3$ not all zero
such that $c_1 A^{(1)} + c_2 A^{(2)} + c_3 A^{(3)} =0$.
Since $\mu $ is reduced, it is not possible that $c_2=c_3 =0$. Hence  $mrk(\mu) \leq2 $  if and only if 
there exists $A \in \mu $ such that 
either $A^{(3)} \in  \langle A^{(1)}, A^{(2)}\rangle $ or $A^{(2)} \in  \langle A^{(1)}, A^{(3)} \rangle $. 

Clearly there exists $A \in \mu $ such that 
$A^{(w)} \in  \langle A^{(1)}, A^{(v)}\rangle $
(with $(v,w) \in\{(2,3),(3,2)\}$)
if and only if at least one of the following holds: 

(1) there exist $\lambda, \gamma \in \R_{\geq 0}$ such that, for any $i,j \in \{1,\dots,p\}$, 
$$ \left\{\begin{array}{l} 
\lambda \, m_{i,1} + \gamma \, m_{i,v} \leq M_{i,w},\\
\lambda \, M_{j,1} + \gamma \, M_{j,v} \geq m_{j,w} , 
\end{array} \right.$$

(2) there exist $\lambda \in \R_{\geq 0}$,
$\gamma \in \R_{\leq 0} $
 such that, for any $i,j \in \{1,\dots,p\}$, 
$$ \left\{\begin{array}{l} 
\lambda \, m_{i,1} + \gamma \, M_{i,v} \leq M_{i,w},\\
\lambda \, M_{j,1} + \gamma \, m_{j,v} \geq m_{j,w},  
\end{array} \right.$$

(3) there exist  $\lambda \in \R_{\leq 0}$,
$\gamma \in \R_{\geq 0} $
 such that, for any $i,j \in \{1,\dots,p\}$, 
$$ \left\{\begin{array}{l} 
\lambda \, M_{i,1} + \gamma \, m_{i,v} \leq M_{i,w},\\
\lambda \, m_{j,1} + \gamma \, M_{j,v} \geq m_{j,w},  
\end{array} \right.$$ 

(4) there exist $\lambda, \gamma \in \R_{\leq 0}$ such that, for any $i,j \in \{1,\dots,p\}$, 
$$ \left\{\begin{array}{l} 
\lambda \, M_{i,1} + \gamma \, M_{i,v} \leq M_{i,w},\\
\lambda \, m_{j,1} + \gamma \, m_{j,v} \geq m_{j,w}.
\end{array} \right.$$ 

Clearly condition (2) is equivalent to 
the existence of  $\lambda, \gamma \in \R_{\geq 0}$
 such that, for any $i,j \in \{1,\dots,p\}$, 
$$ \left\{\begin{array}{l} 
\lambda \, m_{i,1} + \gamma \, (-M_{i,v}) \leq M_{i,w},\\
\lambda \, M_{j,1} + \gamma \,(- m_{j,v}) \geq m_{j,w}.  
\end{array} \right.$$ 
Condition (3) is equivalent to 
the existence of  $\lambda, \gamma \in \R_{\geq 0}$
 such that, for any $i,j \in \{1,\dots,p\}$, 
$$ \left\{\begin{array}{l} 
\lambda \, (-M_{i,1}) + \gamma \, m_{i,v} \leq M_{i,w},\\
\lambda \,(- m_{j,1}) + \gamma \, M_{j,v} \geq m_{j,w}.  
\end{array} \right.$$ 
Finally, 
condition (4) is equivalent to 
the existence of  $\lambda, \gamma \in \R_{\geq 0}$
such that, for any $i,j \in \{1,\dots,p\}$, 
$$ \left\{\begin{array}{l} 
\lambda \,( -M_{i,1}) + \gamma \,(- M_{i,v}) \leq M_{i,w},\\
\lambda \,(- m_{j,1}) + \gamma \,(- m_{j,v}) \geq m_{j,w} . 
\end{array} \right.$$ 
So, by Corollary \ref{corimp}, for $i=1,2,3,4$, condition (i) is equivalent to condition (i) in the statement of the theorem.

\end{proof}

\bigskip

The website http://web.math.unifi.it/users/rubei/ contains the following 
programs (in octave):

(a)  rk01.m: for any interval matrix $\mu$, it 
says if $mrk(\mu) $ is $0$, $1$ or greater than $1$;

(b) mrkpx3.m: for any interval matrix $\mu$ 
with $3$ columns, it calculates $mrk(\mu) $;

(c) Mrk.m: for any interval matrix $\mu$, it calculates $Mrk(\mu) $.

The input $\mu$ must be given as the two matrices $m$ and $M$, containing respectively the minima and the maxima
of the entries of $\mu$.

The programs above use some auxiliary programs, precisely:

$\bullet $ reduce.m: it reduces the interval matrices;

$\bullet$  multisubsets.m: multisubsets(p,k) calculates 
the $k$-mulitsubsets of $\{1, \dots, p\}$ (each is given in lexicographic order);

$\bullet$  orderdmultisubsets.m: 
orderedmultisubsets(p,k) calculates 
the ordered $k$-mulitsubsets of $\{1, \dots, p\}$;

$\bullet$ rk1nnr.m: given a reduced interval matrix $\mu$ with  entries contained in $\R_{\geq 0}$, it
 says whether $\mu$  contains a  matrix with rank $1$;

$\bullet$ sist.m: it says whether a system as in Corollary 
\ref{corimp} is solvable.

$\bullet$ rklesseq2.m:  given a $p \times 3$ reduced interval matrix $\mu $
  with the entries
 of the first column contained in  $\R_{\geq 0}$, the program says 
 whether
  $ mrk(\mu) \leq 2$.

$\bullet $ gooddiag.m: given a $p \times p$  interval matrix $\mu$ and  two permutations 
$I$ and $J$, of $\{1,\dots, p\}$, the program says
if the generalized diagonal of $\mu$ given by $I$ and $J$ is good.

$\bullet $ rkmax.m: given a $p \times p$  interval matrix $\mu$, the program says if $Mrk(\mu)=p$.

The programs above use the theorems and the remarks in the present paper.

\bigskip

{\bf Acknowledgments.}
This work was supported by the National Group for Algebraic and Geometric Structures, and their  Applications (GNSAGA-INdAM).


{\small \begin{thebibliography}{Dilloo Dilloo 83}

%\bibitem{BW} Bostian, A. A.; Woerdeman, H. J. {\em Unicity of minimal rank completions for tri-diagonal partial block matrices.} Linear Algebra Appl. 325 (2001),  23-55.

%\bibitem{BHZ} Brualdi, R. A.; Huang, Z.; Zhan, X. {\em Singular, nonsingular, and bounded rank completions of ACI-matrices.} Linear Algebra Appl. 433 (2010), no. 7, 1452-1462. 

%\bibitem{Cain} Cain, B. E. {\em The inertia of a Hermitian matrix having prescribed diagonal blocks.} Linear Algebra Appl. 37 (1981), 173-180.

%\bibitem{Ca-deSa} Cain, Bryan E.; Marques de Sa, E. {\em  The inertia of a Hermitian matrix having prescribed complementary principal submatrices.} Linear Algebra Appl. 37 (1981), 161-171.


\bibitem{CD} Cohen, N.; Dancis, J.  
{\em Maximal Ranks Hermitian  Completions of Partially specified Hermitian matrices.}
Linear Algebra Appl. 244 (1996), 265-276.

\bibitem{CJRW} Cohen, N.; Johnson, C.R.; Rodman, Leiba; Woerdeman, H. J.  
{\em Ranks of completions of partial matrices.}
The Gohberg anniversary collection, Vol. I (Calgary, AB, 1988), 165-185, 
Oper. Theory Adv. Appl., 40, Birkh\"{a}user, Basel, 1989.

 
%\bibitem{Da} Davis, C.; 
%{\em Completing a matrix so as to minimize the rank.} 
%Topics in operator theory and interpolation, 
%Oper. Theory Adv. Appl. Vol. 29, 87--95, Birkh�user, Basel (1988).
% non ce lo abbiamo da quello che ho trovato in rete le prime due pagine
% sembra risolva le matrici
% A B 
% C ?


%\bibitem{Fie-Mar} Fiedler, M.; Markham, T. L. 
%{\em Rank-preserving diagonal completions of a matrix.} Linear Algebra Appl. 85 (1987), 49-56. 

%\bibitem{Gee} Geelen, J. F. {\em Maximum rank matrix completion.} Linear Algebra Appl. 288 (1999),  211-217.


\bibitem{HHW} Hadwin, D.; Harrison J.; Ward, J.A. {\em
Rank-one completions of partial matrices and 
completing rank-nondecreasing linear functionals.}
Proc. A.M.S. 134  (2006),  no. 8, 2169-2178.

%\bibitem{MQ} McTigue, J.; Quinlan, R. {\em Partial matrices whose completions have ranks bounded below.} Linear Algebra Appl. 435 (2011), no. 9, 2259-2271. 

\bibitem{Rohn} Rohn, J. {\em Systems of Linear Interval Equations.}
Linear Algebra Appl.  126 (1989) 39-78.



\bibitem{Rohn2} Rohn, J. {\em Forty necessary and sufficient conditins for regularity of interval matrices: A survey.}
Electronic Journal of Linear Algebra 18 (2009).



\bibitem{Rohn3} Rohn, J. {\em A Handbook 
of Results on Interval Linear Problems},
Prague: Insitute of Computer Science,
Academy of Sciences of the Czech Republic, 2012.



\bibitem{Rohn4} Rohn, J. {\em
Enclosing solutions of overdetermined systems of linear interval equations}, Reliable Computing, 2 (1996), 167-171.



\bibitem{Shary} Shary, S.P. {\em On 
Full-Rank Interval Matrices} Numerical Analysis and Applications 7 (2014), no. 3, 241-254.

%\bibitem{Tian} Tian, Y. {\em Completing block Hermitian matrices with maximal and minimal ranks and inertias.} Electron. J. Linear Algebra 21 (2010), 124-141.


\bibitem{Woe1}
Woerdeman, H. J.
{\em The lower order of lower triangular operators
and minimal rank extensions.} Integral Equations and Operator Theory 10 (1987), 859-879.

\bibitem{Woe}
Woerdeman, H. J.
{\em Minimal rank completions for block matrices.} 
Linear algebra and applications (Valencia, 1987). Linear Algebra Appl. 
121 (1989), 105-122. 




\end{thebibliography}}
\end{document}